\documentclass[10pt, reqno]{amsart}
\usepackage[font=small, labelfont=bf]{caption}

\usepackage{amsmath, amsfonts, amssymb, amsthm}
\usepackage{graphicx}
\usepackage[arrow, matrix, curve]{xy}
\usepackage{tikz}
\usetikzlibrary{positioning}
\usepackage{booktabs}
\usepackage{color,epsfig}
\usepackage{enumitem}
\usepackage{url}
\usepackage[colorlinks=true, urlcolor=blue, linkcolor=cyan, citecolor=gray]{hyperref}
\usepackage{subcaption}

\allowdisplaybreaks

\theoremstyle{plain}% Theorem-like structures provided by amsthm.sty
\newtheorem{theorem}{Theorem}[section]

\theoremstyle{definition}
\newtheorem{definition}[theorem]{Definition}

\newtheorem{example}[theorem]{Example}

\numberwithin{equation}{section}
\numberwithin{figure}{section}
\numberwithin{table}{section}

\DeclareMathOperator{\Prob}{\operatorname{Pr}}

\DeclareMathOperator{\bbr}{\mathbb{R}}
\DeclareMathOperator{\bbn}{\mathbb{N}}

\DeclareMathOperator{\bfX}{\mathbf{X}}
\DeclareMathOperator{\astbfX}{\mathbf{X}^\ast}
\DeclareMathOperator{\bfx}{\mathbf{x}}
\DeclareMathOperator{\bfY}{\mathbf{Y}}
\DeclareMathOperator{\astbfY}{\mathbf{Y}^\ast}
\DeclareMathOperator{\bfy}{\mathbf{y}}

\DeclareMathOperator{\OPD}{\operatorname{OPD}}
\newcommand{\indicator}[1]{1_{\{ #1 \}}}

\newcommand{\brackets}[1]{\left( #1 \right)}

\definecolor{ao}{rgb}{0.0, 0.5, 0.0}
\newcommand{\revision}[1]{\textcolor{black}{#1}}

\begin{document}

\title{Ordinal pattern dependence and multivariate measures of dependence}

\author[A. Silbernagel]{Angelika Silbernagel} 
\author[A. Schnurr]{Alexander Schnurr}

\address{Department of Mathematics, University of Siegen, 57072 Siegen, Germany}
\email{silbernagel@mathematik.uni-siegen.de}

\begin{abstract}
Ordinal pattern dependence has been introduced in order to capture co-monotonic behavior between two time series. This concept has several features one would intuitively demand from a dependence measure. It was believed that ordinal pattern dependence satisfies the axioms which Grothe et al. \revision{(\textit{Journal of Multivariate Analysis} 123, 2014)} proclaimed for a multivariate measure of dependence. In the present article we show that this is not true and that there is a mistake in the article by Betken et al. \revision{(\textit{Journal of Multivariate Analysis} 186, 2021)}. Furthermore\revision{,} we show that ordinal pattern dependence satisfies a slightly modified set of axioms.

\vspace{.3cm}
\end{abstract}
\maketitle

\section{Introduction}

Since their introduction in the seminal paper by Bandt and Pompe \cite{ban_pom_02}, ordinal patterns have been widely applied in the context of time series analysis, and in particular in the context of dependence. 
For example, ordinal patterns have been used to uncover serial dependence in 
\cite{ban_23}, \cite{wei_schn_23} and \cite{wei_rui_kel_22}.
With regard to dependence \emph{between} time series, ordinal pattern dependence, which has been proposed in \cite{schn_14}, 
\revision{captures how} strong the co-movement between two data sets or two time series is. 
Applications of ordinal pattern dependence include, but are not limited to, hydrological data (cf. \cite{schn_fis_22b}, \cite{schn_fis_22a}), health (cf. \cite{wei_schn_23}) and finance (cf. \cite{schn_14}).
Moreover, limit theorems have been provided in \cite{schn_deh_17} and \cite{bet_buch_deh_mue_schn_woe_21} with regard to the short-range and long-range dependent cases, respectively. \revision{In \cite{schn_fis_22b} limit theorems were proven} under short-range dependence with regard to generalized ordinal patterns.

\revision{A question of interest is} how ordinal pattern dependence is related to other dependence measures. In \cite{betkenetal} the authors proved interesting results on the relationship towards e.g. Kendall's tau and Pearson's correlation coefficient. Furthermore, they aimed to show that ordinal pattern dependence fits into the axiomatic framework for multivariate measures of dependence between random vectors of same dimension which was suggested in \cite{grotheetal}.  However, there is an error in the proof of the fifth axiom. In general, this axiom cannot be verified for ordinal pattern dependence.  In Section \ref{section: counterexample} we provide a suitable counterexample. While the first 4 axioms in \cite{grotheetal} are (more or less) canonical, the fifth one is strongly inspired by \cite{schm_schm_blu_gai_rup_10} (which deals with dependence \emph{within} one random vector). Arguably one could have formulated this axiom differently. In  Section \ref{section: proof} we give a proof for the fifth axiom making different assumptions with regard to the (conditional) cumulative distribution functions and (conditional) survival functions, i.e., we make an alteration to the assumptions on concordance ordering.

\section{Preliminaries}

Let $(X_t, Y_t)_{t \in \bbn}$ be a stationary, bivariate time series with continuous marginal distributions. For fixed $d \geq 2$, we denote the random vectors of consecutive observations by 
\begin{align*}
    \bfX_t &:= (X_t, \ldots, X_{t+d-1}) \\
    \bfY_t &:= (Y_t, \ldots, Y_{t+d-1}).
\end{align*}
Due to the stationarity of the underlying process, the joint distribution of $(\bfX_t, \bfY_t)$ does not depend on the time index $t$, hence, it is sufficient to consider the random vector $(\bfX_1, \bfY_1)$, which, for simplicity, we denote by $(\bfX, \bfY)$. 

\begin{definition}\label{def: op rank}
    Let $S_d$ denote the set consisting of the $d!$ possible permutations of the integers $\{1, \ldots, d\}$. The \emph{ordinal pattern of order $d$} of a vector $\bfx = (x_1, \ldots, x_d) \in \bbr^d$ is defined as the permutation $\pi = (\pi_1, \ldots, \pi_d) \in S_d$ which satisfies 
    \begin{equation}
        \pi_i < \pi_j \Leftrightarrow x_i < x_j \text{ or } (x_i = x_j \text{ and } i < j) \label{eq: ordinal pattern}
    \end{equation}
    for all $i, j \in \{1, \ldots, d\}$.
\end{definition}

By $\Pi : \bbr^d \rightarrow S_d$ we denote the map which assigns the vector $\bfx$ with its corresponding pattern, that is, $\Pi(\bfx) = \pi$ satisfies \eqref{eq: ordinal pattern}. %In order to avoid ties, i.e., equalities between components of $\bfx$, we assume from now on that all random vectors under consideration admit continuous distribution functions. 
Even though Def. \ref{def: op rank} is different from the definition given in \cite{betkenetal}, they are equivalent in the sense that the respective permutations which describe the ordinal pattern at hand can be converted into each other via some permutation function (cf. \cite{schn_sil_23}). The advantage of the definition above is that it is more intuitive, since there one can interpret the entries of $\pi$ as ranks.
The value in the subsequent definition does not depend on the choice of the representation.

\begin{definition}
    The \emph{ordinal pattern dependence} between two random vectors $\bfX := (X_1, \ldots, X_d)$ and $\bfY := (Y_1, \ldots, Y_d)$ is defined by 
    \begin{equation*}
            \OPD_d(\bfX, \bfY) 
            := \frac{\Prob(\Pi(\bfX) = \Pi(\bfY)) - \sum_{\pi \in S_d} \Prob(\Pi(\bfX) = \pi) \Prob(\Pi(\bfY) = \pi)}{1 - \sum_{\pi \in S_d} \Prob(\Pi(\bfX) = \pi) \Prob(\Pi(\bfY) = \pi)}.
    \end{equation*}
\end{definition}

Intuitively speaking, one compares in the numerator the probability of coincident patterns with the hypothetical case of independence. This value is normed by the denominator. \revision{For any random vector $\bfX = (X_1, \ldots, X_d)$, let 
\begin{equation*}
    F_{\bfX}(\bfx) := \Prob(\bfX \leq \bfx) = \Prob(X_1 \leq x_1, \ldots, X_d \leq x_d) 
\end{equation*}
and
\begin{equation*}
    \overline{F}_{\bfX}(\bfx) := \Prob(\bfX \geq \bfx) = \Prob(X_1 \geq x_1, \ldots, X_d \geq x_d) 
\end{equation*}
denote the joint cumulative distribution function (cdf) and survival function, respectively. Note that for $d \geq 2$ in general $\overline{F}_{\bfX}(\bfx) \neq 1 - F_{\bfX}(\bfx)$.}
With the following definition, \cite{grotheetal} proposed an axiomatic theory for multivariate measures of dependence between random vectors of same dimension. For the \revision{readers'} convenience, we state it in the wording of \cite{betkenetal}:

\begin{definition}
    Let $L_0$ denote the space of random vectors with values in $\bbr^d$ on the common probability space $(\Omega, \mathcal{F}, \Prob)$. A function $\mu : L_0 \times L_0 \rightarrow \bbr$ is called a \emph{$d$-dimensional measure of dependence} if
    \begin{enumerate}
        \item it takes values in $[-1, 1]$,
        \item it is invariant with respect to simultaneous permutations of the components within two random vectors $\bfX, \bfY$,
        \item it is invariant with respect to increasing transformations of the components within two random vectors $\bfX, \bfY$,
        \item it is zero for two independent random vectors $\bfX, \bfY$,
        \item it respects concordance ordering, i.e., for two pairs of random vectors $\bfX, \bfY$ and $\astbfX, \astbfY$ which satisfy $\bfX \overset{D}{=} \astbfX$ and $\bfY \overset{D}{=} \astbfY$, it holds that 
        \begin{equation*}
            \brackets{\begin{array}{c}
                \bfX \\
                \bfY 
            \end{array}}
            \preccurlyeq_C
            \brackets{\begin{array}{c}
                \astbfX \\
                \astbfY 
            \end{array}}
            \Rightarrow \mu(\bfX, \bfY) \leq \mu(\astbfX, \astbfY).
        \end{equation*}
        Here, $\preccurlyeq_C$ denotes concordance ordering, i.e., 
        \begin{equation*}
            \brackets{\begin{array}{c}
                \bfX \\
                \bfY 
            \end{array}}
            \preccurlyeq_C
            \brackets{\begin{array}{c}
                \astbfX \\
                \astbfY 
            \end{array}}
            \text{ if and only if } 
            F_{\brackets{\substack{\bfX\\\bfY}}} \leq F_{\brackets{\substack{\astbfX\\\astbfY}}}
            \text{ and }
            \overline{F}_{\brackets{\substack{\bfX\\\bfY}}} \leq \overline{F}_{\brackets{\substack{\astbfX\\\astbfY}}},
        \end{equation*}
        \revision{where $\leq$ is meant pointwise.}
    \end{enumerate}
\end{definition}

In \cite{betkenetal} it was claimed that ordinal pattern dependence $\OPD_d$ is a $d$-dimensional measure of dependence (cf. \cite{betkenetal} Theorem 2.3). While the proof of the first four axioms is correct, there is an error regarding the treatment of the conditional probabilities and concordance ordering in the proof of the fifth axiom.

\section{A Counterexample}\label{section: counterexample}

We now give a counterexample showing that ordinal pattern dependence cannot satisfy the fifth axiom and thus does not fit into the axiomatic framework of multivariate measures of dependence as proposed by \cite{grotheetal}. 
As stated in \cite{betkenetal}, for the verification of the last axiom it is sufficient to restrict our considerations to the quantity $\Prob(\Pi(\bfX) = \Pi(\bfY))$, since the remaining terms of $\OPD_d(\bfX, \bfY)$ only relate to the distribution of $\bfX$ and $\bfY$ separately. (Note that $\bfX \overset{D}{=} \astbfX$ and $\bfY \overset{D}{=} \astbfY$.) For $d=2$, there are only two patterns, namely the upward pattern (1, 2) and the downward pattern (2, 1). Even in this case, the problem is 4-dimensional, since we have to deal with vectors of length $2d$ when considering dependence.

In practice ordinal pattern dependence is often considered in the context of stationary time series. This can be broken down to the  conditions $X_1 \overset{D}{=} X_2$ and $Y_1 \overset{D}{=} Y_2$ in our case. Our counterexample satisfies this additional property.

\begin{example}
We consider the probability densities $f = f_{\brackets{\substack{\bfX\\\bfY}}}$ and $f^\ast = f_{\brackets{\substack{\astbfX\\\astbfY}}}$ defined by \revision{
\begin{align*}
    f(x_1, x_2, y_1, y_2) &= \indicator{0 \leq x_1 \leq x_2 \leq 1} \cdot \indicator{1 < y_1 \leq y_2 \leq 2} 
    + \indicator{0 \leq x_2 < x_1 \leq 1} \cdot \indicator{1 < y_2 < y_1 \leq 2} \\
    &\hspace{10mm} + \indicator{1 < x_1 \leq x_2 \leq 2} \cdot \indicator{0 \leq y_1 \leq y_2 \leq 1}
    + \indicator{1 < x_2 < x_1 \leq 2} \cdot \indicator{0 \leq y_2 < y_1 \leq 1} \\
    f^\ast(x_1, x_2, y_1, y_2) &= \indicator{0 \leq x_1 \leq x_2 \leq 1} \cdot \indicator{0 \leq y_2 < y_1 \leq 1} 
    + \indicator{0 \leq x_2 < x_1 \leq 1} \cdot \indicator{1 < y_2 < y_1 \leq 2} \\
    &\hspace{10mm} + \indicator{1 < x_1 \leq x_2 \leq 2} \cdot \indicator{0 \leq y_1 \leq y_2 \leq 1}
    + \indicator{1 < x_2 < x_1 \leq 2} \cdot \indicator{1 < y_1 \leq y_2 \leq 2}.
\end{align*}
Obviously it holds 
\begin{align*}
    f_{\bfX}(x_1, x_2) 
    &= \frac{1}{2} \brackets{\indicator{0 \leq x_1 \leq x_2 \leq 1} + \indicator{0 \leq x_2 < x_1 \leq 1}} 
    + \frac{1}{2} \brackets{\indicator{1 < x_1 \leq x_2 \leq 2} + \indicator{1 < x_2 < x_1 \leq 2}} \\
    &= \frac{1}{2} \brackets{\indicator{0 \leq x_1, x_2 \leq 1} + \indicator{1 < x_1, x_2 \leq 2}} \\
    &= f_{\astbfX}(x_1, x_2)
\end{align*}
as well as 
\begin{equation*}
    f_{X_1}(x) = \frac{1}{2} \brackets{\indicator{0 \leq x \leq 1} + \indicator{1 < x \leq 2}} = \frac{1}{2} \indicator{0 \leq x \leq 2} = f_{X_2}(x)
\end{equation*}
such that $\bfX \overset{D}{=} \astbfX$ and $X_1 \overset{D}{=} X_2$.} Analogously it follows $\bfY \overset{D}{=} \astbfY$ and $Y_1 \overset{D}{=} Y_2$. \\
In order to prove the conditions on the joint cdfs and survival functions, 
\revision{first we take a closer look at the respective density functions and observe that the summands $\indicator{0 \leq x_2 < x_1 \leq 1} \cdot \indicator{1 < y_2 < y_1 \leq 2}$ and $\indicator{1 < x_1 \leq x_2 \leq 2} \cdot \indicator{0 \leq y_1 \leq y_2 \leq 1}$ appear in both $f$ and $f^\ast$. Therefore, for our purpose it is sufficient to only consider the auxiliary functions $h$ and $h^\ast$ defined by
\begin{align*}
    h(x_1, x_2, y_1, y_2) &= \indicator{0 \leq x_1 \leq x_2 \leq 1} \cdot \indicator{1 < y_1 \leq y_2 \leq 2} 
    + \indicator{1 < x_2 < x_1 \leq 2} \cdot \indicator{0 \leq y_2 < y_1 \leq 1} \\
    h^\ast(x_1, x_2, y_1, y_2) &= \indicator{0 \leq x_1 \leq x_2 \leq 1} \cdot \indicator{0 \leq y_2 < y_1 \leq 1} 
    + \indicator{1 < x_2 < x_1 \leq 2} \cdot \indicator{1 < y_1 \leq y_2 \leq 2}.
\end{align*}
Now, let us
}
consider Fig. \ref{fig: stationary counterexample auxiliary fcts}. There, it is depicted where the respective \revision{auxiliary functions $h$ and $h^\ast$} hold mass, where blue denotes the case $0 \leq x_1 \leq x_2 \leq 1$, while purple denotes the case \revision{$1 < x_2 < x_1 \leq 2$}. Note that each \revision{auxiliary function} contains both cases, and only these two cases in particular. Hence, it is sufficient to limit our considerations with regard to the comparison of the respective cdfs and survival functions to these two scenarios. \\
With regard to the cdfs, in case of $0 \leq x_1 \leq 1$ \textbf{or} $0 \leq x_2 \leq 1$ we only consider the mass denoted by the blue triangles, while we need to consider the triangles of both colors in case of $x_1, x_2 > 1$. Then, for both scenarios rectangles of the form $]-\infty, y_1] \times ]-\infty, y_2]$ contain more mass with regard to \revision{$h^\ast$} if compared to \revision{$h$}, so $F_{\brackets{\substack{\bfX\\\bfY}}} \leq F_{\brackets{\substack{\astbfX\\\astbfY}}}$ pointwise. In contrast, with regard to survival functions we need to proceed `the other way around': For $0 \leq x_1 \leq 1$ \textbf{or} $0 \leq x_2 \leq 1$ we only consider the purple triangles. In case of $x_1, x_2 > 1$ it is both colors again. Hence, for rectangles of the form $[y_1, \infty[ \times [y_2, \infty[$, \revision{$h^\ast$} holds more mass, and therefore $\overline{F}_{\brackets{\substack{\bfX\\\bfY}}} \leq \overline{F}_{\brackets{\substack{\astbfX\\\astbfY}}}$ pointwise, which shows that all conditions are satisfied. \\
\revision{Now, we return to considering the density functions $f$ and $f^\ast$. By construction of $f$, an increasing pattern $x_1 \leq x_2$ in the first component occurs if and only if it also occurs in the second component, that is $y_1 \leq y_2$. The same holds true for decreasing patterns. However, the situation is different with regard to $f^\ast$: There, coincident patterns are only predetermined by two of the four summands. The remaining two summands only allow for non-coincident patterns.}
Then, due to
\begin{equation*}
    \Prob(\Pi(\bfX) = \Pi(\bfY)) = 1 > \revision{\frac{1}{2}} = \Prob(\Pi(\astbfX) = \Pi(\astbfY))
\end{equation*}
it follows $\OPD_2(\bfX, \bfY) > \OPD_2(\astbfX, \astbfY)$. Thus, we have found two vectors which satisfy the property of concordance ordering showing the opposite behavior in terms of ordinal pattern dependence.
\end{example}

\begin{figure}[ht]
    \centering
    \begin{subfigure}{0.45\textwidth}
        \centering
        \begin{tikzpicture}[scale=1.2]
            \fill[green!5] (-1, -1) rectangle (0.5, 1.5);
            \filldraw[blue!10] (1, 1) -- (1, 2) -- (2, 2) -- (1, 1);
            \filldraw[purple!10] (0, 0) -- (1, 1) -- (1, 0) -- (0, 0);
            
            \draw[olive] (-1, 1.5) -- (0.5, 1.5) -- (0.5, -1);
            \filldraw[black] (0.5, 1.5) circle (1pt) node[anchor=west]{(0.5, 1.5)};
            
            \draw[->] (0, 0) -- (3, 0);
            \draw[->] (0, 0) -- (0, 3);
            \node at (3.2, -0.3) {$y_1$};
            \node at (-0.3, 3) {$y_2$};
            \foreach \x in {1, 2}{
                \draw (\x,0) -- (\x,-0.1);
                \node at (\x,-0.3) {$\x$};
                \draw (0, \x) -- (-0.1, \x);
                \node at (-0.3, \x) {$\x$};
            };
        \end{tikzpicture}
        \caption{}
    \end{subfigure}
    \begin{subfigure}{0.45\textwidth}
        \centering
        \begin{tikzpicture}[scale=1.2]
            \fill[green!5] (-1, -1) rectangle (0.5, 1.5);
            \filldraw[blue!10] (0, 0) -- (1, 1) -- (1, 0) -- (0, 0);
            \filldraw[purple!10] (1, 1) -- (1, 2) -- (2, 2) -- (1, 1);

            \draw[olive] (-1, 1.5) -- (0.5, 1.5) -- (0.5, -1);
            \filldraw[black] (0.5, 1.5) circle (1pt) node[anchor=west]{(0.5, 1.5)};
            
            \draw[->] (0, 0) -- (3, 0);
            \draw[->] (0, 0) -- (0, 3);
            \node at (3.2, -0.3) {$y_1$};
            \node at (-0.3, 3) {$y_2$};
            \foreach \x in {1, 2}{
                \draw (\x,0) -- (\x,-0.1);
                \node at (\x,-0.3) {$\x$};
                \draw (0, \x) -- (-0.1, \x);
                \node at (-0.3, \x) {$\x$};
            };
        \end{tikzpicture}
        \caption{}
    \end{subfigure}
    \caption{Illustration of the \revision{mass of $h$ (left) and $h^\ast$ (right)}: The functions attain the value \revision{1} for $0 \leq x_1 \leq x_2 \leq 1$ in the blue area and for \revision{$1 < x_2 < x_1 \leq 2$} in the purple area. Otherwise they are zero. Consider the rectangle $]-\infty, 0.5] \times ]-\infty, 1.5]$ (green) as an illustrative example with regard to the respective cumulative distribution functions.}
    \label{fig: stationary counterexample auxiliary fcts}
\end{figure}
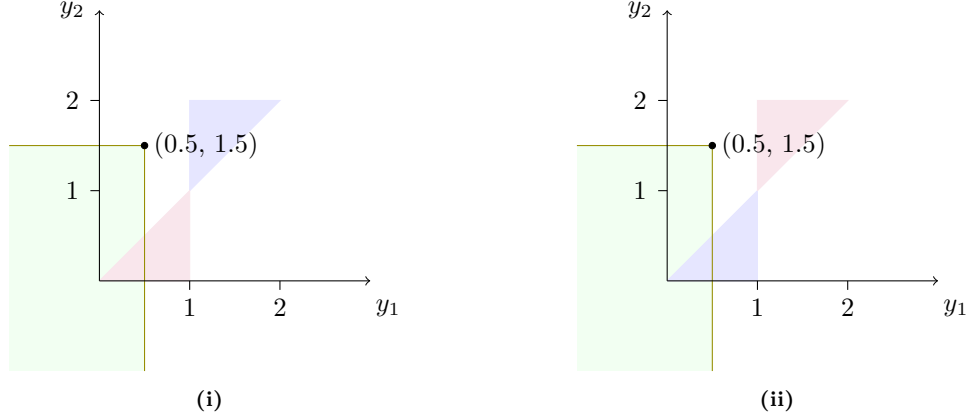

%%%%%%%%%%%%%%%%%%
\section{Proof under slightly different assumptions}\label{section: proof}

Although ordinal pattern dependence does not satisfy the fifth axiom and hence does not fit into the axiomatic framework of multivariate measures of dependence as proposed by \cite{grotheetal}, we can prove a similar result under stronger/different assumptions.

\begin{theorem}\label{theo: OPD axiom 5 diff assump}
    Let $\bfX, \astbfX, \bfY$ and $\astbfY$ be $d$-dimensional random vectors with $\bfX \overset{D}{=} \astbfX$ and $\bfY \overset{D}{=} \astbfY$ Let $J = \{1, \ldots, d\}$ denote an index set.
    \newline (a) Suppose for all $I \subset J$, $I^c := J \setminus I$, with $I, I^c \neq \emptyset$ it holds
    \begin{equation*}
         F_{\left.\brackets{\substack{\bfX\\\bfY}}^{I^c}\right|\brackets{\substack{\bfX\\\bfY}}^I} 
        \leq 
        F_{\left.\brackets{\substack{\astbfX\\\astbfY}}^{I^c}\right|\brackets{\substack{\bfX\\\bfY}}^I}
        \hspace{1cm} \text{ and } \hspace{1cm}
        \overline{F}_{\left.\brackets{\substack{\bfX\\\bfY}}^{I^c}\right|\brackets{\substack{\bfX\\\bfY}}^I} 
        \leq 
        \overline{F}_{\left.\brackets{\substack{\astbfX\\\astbfY}}^{I^c}\right|\brackets{\substack{\bfX\\\bfY}}^I}
        %\label{eq: OPD assump 1}
    \end{equation*}
    as well as
    \begin{equation*}
        F_{\left.\brackets{\substack{\bfX\\\bfY}}^{I^c}\right|\brackets{\substack{\astbfX\\\astbfY}}^I} 
        \leq 
        F_{\left.\brackets{\substack{\astbfX\\\astbfY}}^{I^c}\right|\brackets{\substack{\astbfX\\\astbfY}}^I}
        \hspace{1cm}\text{ and } \hspace{1cm}
        \overline{F}_{\left.\brackets{\substack{\bfX\\\bfY}}^{I^c}\right|\brackets{\substack{\astbfX\\\astbfY}}^I} 
        \leq 
        \overline{F}_{\left.\brackets{\substack{\astbfX\\\astbfY}}^{I^c}\right|\brackets{\substack{\astbfX\\\astbfY}}^I}
        %\label{eq: OPD assump 2}
    \end{equation*}
    pointwise, where $\brackets{\substack{\bfX\\\bfY}}^I = (\bfX^I, \bfY^I)$ denotes the subvector of variables with indices in $I$. Then, it holds that 
    \begin{equation} 
        \OPD_d(\bfX, \bfY) \leq \OPD_d(\astbfX, \astbfY). \label{eq: OPD result}
    \end{equation}
    (b) Alternatively, the condition 
    \begin{equation}
        F_{\left.\brackets{\substack{\bfX\\\bfY}}^{I^c}\right|\brackets{\substack{\bfX\\\bfY}}^I} 
        \leq 
        F_{\left.\brackets{\substack{\astbfX\\\astbfY}}^{I^c}\right|\brackets{\substack{\bfX\\\bfY}}^I}
        \hspace{1cm}\text{ and } \hspace{1cm}
        \overline{F}_{\left.\brackets{\substack{\bfX\\\bfY}}^{I^c}\right|\brackets{\substack{\bfX\\\bfY}}^I} 
        \leq 
        \overline{F}_{\left.\brackets{\substack{\astbfX\\\astbfY}}^{I^c}\right|\brackets{\substack{\bfX\\\bfY}}^I} \label{eq: OPD assump 3}
    \end{equation}
    for all $I \subset J$ with $I \neq J$ is sufficient for \eqref{eq: OPD result} to hold. In particular, with $I = \emptyset$ we make a statement about the (unconditional) cdf and survival function.
\end{theorem}

Note that Theorem %\ref{theo: Dudley th10.2.2} 
10.2.2 of \cite{dudley}
ensures existence and uniqueness of the regular conditional distributions used above, such that the (conditional) cdfs and survival functions are defined for almost all $(\bfx, \bfy) \in \bbr^{2 \cdot |I|}$ by definition, respectively. \revision{This statement on regular conditional distributions is important as in general consideration of conditional probabilities where the event that is conditioned on has probability mass zero can be problematic.} The proof of Theorem \ref{theo: OPD axiom 5 diff assump} follows the ideas of \cite{betkenetal} and hence, makes use of the so called disintegration theorem. Basically, disintegration denotes the representation of a (conditional) expectation as an integral in terms of a (regular conditional) probability distribution (for more information, see, e.g., \cite{dudley} Chapter 10). Note further that in general there is neither equivalence nor any implication between the two sets of assumptions in Theorem \ref{theo: OPD axiom 5 diff assump}.

\begin{proof}[Proof of Theorem \ref{theo: OPD axiom 5 diff assump}]
    We give a proof for $d=2$ and $d=3$. Though the difficulties of the proof are not revealed for $d=2$, it gives an intuition how to proceed in case $d=3$. The proof for $d>3$ works analogously to the case $d=3$, but is notationally more complicated. \\
    Due to stationarity, it is sufficient to consider the first components of the time series, i.e., without loss of generality we consider
    \begin{equation*}
    \begin{split}
        &\bfX = (X_1, \ldots, X_d), \bfY = (Y_1, \ldots, Y_d), \\ 
        &\astbfX = (X^\ast_1, \ldots, X^\ast_d), \astbfY = (Y^\ast_1, \ldots, Y^\ast_d).
    \end{split}
    \end{equation*}
    Furthermore, since the remaining summands of $\OPD_d(\bfX, \bfY)$ only relate to the distributions of $\bfX$ and $\bfY$ separately, we restrict our considerations to \revision{$\Prob(\Pi(\bfX) = \Pi(\bfY)) = \sum_{\pi \in S_d} \Prob(\Pi(\bfX) = \Pi(\bfY) = \pi)$. Now, without loss of generality, it is sufficient to only consider the probability $\Prob(\Pi(\bfX) = \Pi(\bfY) = \pi)$ with regard to \textbf{one} arbitrary pattern $\pi \in S_d$, as for any other pattern $\pi' \neq \pi$ there exists a permutation $\sigma : \{1, \ldots, d\} \to \{1, \ldots, d\}$ such that 
    \begin{equation*}
        \Prob(\Pi(X_1, \ldots, X_d) = \Pi(Y_1, \ldots, Y_d) = \pi') 
        = \Prob(\Pi(X_{\sigma(1)}, \ldots, X_{\sigma(d)}) = \Pi(Y_{\sigma(1)}, \ldots, Y_{\sigma(d)}) = \pi).
    \end{equation*}
    Hence, everything can be reduced to a re-indexing such that considering these probabilities for the respective patterns is analoguous. Here, we choose the increasing patterns (1, 2) and (1, 2, 3), respectively.} \\
    We denote the conditional cdfs and survival functions by
    \begin{align*}
        \overline{F}^{x_{i_1}, \ldots, x_{i_k}, y_{i_1}, \ldots, y_{i_k}}_{\brackets{\substack{\bfX\\\bfY}}^{I^c}}(\bfx, \bfy) 
        &:= \overline{F}^{x_{i_1}, \ldots, x_{i_k}, y_{i_1}, \ldots, y_{i_k}}_{\left.\brackets{\substack{\bfX\\\bfY}}^{I^c}\right|\brackets{\substack{\bfX\\\bfY}}^I}(\bfx, \bfy) \\
        &:= \Prob\brackets{(\bfX, \bfY)^{I^c} \geq (\bfx, \bfy) | X_{i_j} = x_{i_j}, Y_{i_j} = y_{i_j} \forall j \in \{1, \ldots, k\}} 
        \\
        \overline{F}^{x_{i_1}, \ldots, x_{i_k}, y_{i_1}, \ldots, y_{i_k}}_{\brackets{\substack{\astbfX\\\astbfY}}^{I^c}}(\bfx, \bfy) 
        &:= \overline{F}^{x_{i_1}, \ldots, x_{i_k}, y_{i_1}, \ldots, y_{i_k}}_{\left.\brackets{\substack{\astbfX\\\astbfY}}^{I^c}\right|\brackets{\substack{\bfX\\\bfY}}^I}(\bfx, \bfy) \\
        &:= \Prob\brackets{(\astbfX, \astbfY)^{I^c} \geq (\bfx, \bfy) | X_{i_j} = x_{i_j}, Y_{i_j} = y_{i_j} \forall j \in \{1, \ldots, k\}} 
        \\
        \revision{ F^{x_{i_1}, \ldots, x_{i_k}, y_{i_1}, \ldots, y_{i_k}}_{\left.\brackets{\substack{\bfX\\\bfY}}^{I^c}\right|\brackets{\substack{\astbfX\\\astbfY}}^I}(\bfx, \bfy)} 
        &\revision{\ := \Prob\brackets{(\bfX, \bfY)^{I^c} \leq (\bfx, \bfy) | X^\ast_{i_j} = x_{i_j}, Y^\ast_{i_j} = y_{i_j} \forall j \in \{1, \ldots, k\}} } 
        \\
        \revision{ F^{x_{i_1}, \ldots, x_{i_k}, y_{i_1}, \ldots, y_{i_k}}_{\left.\brackets{\substack{\astbfX\\\astbfY}}^{I^c}\right|\brackets{\substack{\astbfX\\\astbfY}}^I}(\bfx, \bfy)} 
        &\revision{\ := \Prob\brackets{(\astbfX, \astbfY)^{I^c} \leq (\bfx, \bfy) | X^\ast_{i_j} = x_{i_j}, Y^\ast_{i_j} = y_{i_j} \forall j \in \{1, \ldots, k\}} }
    \end{align*}
    for all $\revision{I = } \ \{i_1, \ldots, i_k\} \subset \{1, \ldots, d\}$ with $1 \leq k \leq d-1$, \revision{$I^c := \{1, \ldots, d\} \setminus I$} and $(\bfx, \bfy) \in \bbr^{2(d-k)}$. \\
    Suppose $d=2$. We begin with the proof of (a) in this setting. Using disintegration as stated in Theorem 
    %\ref{theo: Dudley th10.2.5}
    10.2.5 of \cite{dudley}
    (or alternatively the law of total expectation), it holds
    \begin{align*}
        \Prob(X_1 \leq X_2, Y_1 \leq Y_2) 
        &= \int_{\bbr^2} \Prob(X_1 \leq X_2, Y_1 \leq Y_2 | X_1 = x_1, Y_1 = y_1)  \ d\Prob_{\brackets{\substack{X_1\\Y_1}}}(x_1, y_1) \\
        &= \int_{\bbr^2} \Prob(x_1 \leq X_2, y_1 \leq Y_2 | X_1 = x_1, Y_1 = y_1)  \ d\Prob_{\brackets{\substack{X_1\\Y_1}}}(x_1, y_1) \\
        &= \int_{\bbr^2} \overline{F}_{\brackets{\substack{X_2\\Y_2}}}^{x_1, y_1}(x_1, y_1)  \ d\Prob_{\brackets{\substack{X_1\\X_2}}}(x_1, y_1).
    \end{align*}
    Using our assumption and reversing the previous steps yields
    \begin{align*}
        \int_{\bbr^2} \overline{F}_{\brackets{\substack{X_2\\Y_2}}}^{x_1, y_1}(x_1, y_1)  \ d\Prob_{\brackets{\substack{X_1\\Y_1}}}(x_1, y_1) 
        &\leq \int_{\bbr^2} \overline{F}_{\brackets{\substack{X^\ast_2\\Y^\ast_2}}}^{x_1, y_1}(x_1, y_1)  \ d\Prob_{\brackets{\substack{X_1\\Y_1}}}(x_1, y_1) \\
        &= \Prob(X_1 \leq X^\ast_2, Y_1 \leq Y^\ast_2).
    \end{align*}
    In an analogous way we obtain
    \begin{align*}
        \Prob(X_1 \leq X^\ast_2, Y_1 \leq Y^\ast_2)
        = &\int_{\bbr^2} \revision{F_{\left.\brackets{\substack{X_1\\Y_1}}\ \right| \ \brackets{\substack{X^\ast_2\\Y^\ast_2}}}^{x_2, y_2}(x_2, y_2)} \ d\revision{\Prob_{\brackets{\substack{X^\ast_2\\Y^\ast_2}}}(x_2, y_2)} \\
        \leq &\int_{\bbr^2} \revision{F_{\left.\brackets{\substack{X^\ast_1\\Y^\ast_1}} \ \right| \ \brackets{\substack{X^\ast_2\\Y^\ast_2}}}^{x_2, y_2}(x_2, y_2)} \ d\revision{\Prob_{\brackets{\substack{X^\ast_2\\Y^\ast_2}}}(x_2, y_2)} \\
        = &\Prob(X^\ast_1 \leq X^\ast_2, Y^\ast_1 \leq Y^\ast_2).
    \end{align*}
    Under the assumptions in (b) the desired inequality follows due to the fact that 
    \begin{equation*}
        \brackets{\begin{array}{c}
            \bfX \\
            \bfY 
        \end{array}}
        \preccurlyeq_C
        \brackets{\begin{array}{c}
            \astbfX \\
            \astbfY 
        \end{array}}
        \text{ implies }
        \brackets{\begin{array}{c}
            \bfX \\
            \bfY 
        \end{array}}^{I'}
        \preccurlyeq_C
        \brackets{\begin{array}{c}
            \astbfX \\
            \astbfY 
        \end{array}}^{I'}
    \end{equation*}
    for all subvectors of variables with indices in $I' \subset J$, i.e., removing dimensions does not influence which scenario has the larger dependence measure (\cite{grotheetal}). Note that 
    $\brackets{\begin{array}{c}
            \bfX \\
            \bfY 
        \end{array}}
        \preccurlyeq_C
        \brackets{\begin{array}{c}
            \astbfX \\
            \astbfY 
        \end{array}}$
    holds due to Equality \eqref{eq: OPD assump 3} with $I = \emptyset$. Hence, we deduce that 
    \begin{equation}
        \int 1_{[a, \infty[\times[b, \infty[}(x, y) \ d\Prob_{\brackets{\substack{X_i\\Y_i}}}(x, y) 
        \leq 
        \int 1_{[a, \infty[\times[b, \infty[}(x, y) \ d\Prob_{\brackets{\substack{X^\ast_i\\Y^\ast_i}}}(x, y) \label{eq: uncond ineq indicator fct}
    \end{equation}
    for all $a, b \in \bbr$ and $i\in\{1, \ldots, d\}$. Survival functions can be approximated by sums of indicator functions, \revision{i.e., $\overline{F}(x,y) = \lim_{n \to \infty} \sum^n_{k=1} \indicator{[a_k, \infty) \times[b_k, \infty)}(x,y)$ for constants $a_k, b_k \in \bbr$,} so the bounded convergence theorem yields
    \begin{align*}
        \int_{\bbr^2} \overline{F}_{\brackets{\substack{X^\ast_2\\Y^\ast_2}}}^{x_1, y_1}(x_1, y_1) \ d\Prob_{\brackets{\substack{X_1\\Y_1}}}(x_1, y_1)
        &\leq 
        \int_{\bbr^2} \overline{F}_{\brackets{\substack{X^\ast_2\\Y^\ast_2}}}^{x_1, y_1}(x_1, y_1) \ d\Prob_{\brackets{\substack{X^\ast_1\\Y^\ast_1}}}(x_1, y_1) \\
        &= \Prob(X^\ast_1 \leq X^\ast_2, Y^\ast_1 \leq Y^\ast_2).
    \end{align*}
    Now, suppose $d=3$. Defining
    \begin{equation*}
        \Prob^{x_1, y_1}(A) := \Prob(A|X_1 = x_1, Y_1 = y_1)
    \end{equation*}
    for any event $A$, and using disintegration (Theorem 
    %\ref{theo: Dudley th10.2.5})
    10.2.5 of \cite{dudley})
    twice yields
    \begin{align*}
        &\Prob(X_1 \leq X_2 \leq X_3, Y_1 \leq Y_2 \leq Y_3) \\
        & \hspace{10mm} = \int_{\bbr^2} \Prob^{x_1, y_1}(x_1 \leq X_2 \leq X_3, y_1 \leq Y_2 \leq Y_3)\ d\Prob_{\brackets{\substack{X_1\\Y_1}}}(x_1, y_1) \\
        & \hspace{10mm} = \int_{\bbr^2} \int_{[x_1, \infty[ \times [y_1, \infty[} \Prob^{x_1, y_1}(x_2 \leq X_3, y_2 \leq Y_3 | X_2 = x_2, Y_2 = y_2) \ d\Prob^{x_1, y_1}_{\brackets{\substack{X_2\\Y_2}}}(x_2, y_2) \ d\Prob_{\brackets{\substack{X_1\\Y_1}}}(x_1, y_1).
    \end{align*}
    Since 
    \begin{align*}
        &\Prob^{x_1, y_1}(x_2 \leq X_3, y_2 \leq Y_3 | X_2 = x_2, Y_2 = y_2) \\
        & \hspace{15mm} = \Prob(x_2 \leq X_3, y_2 \leq Y_3 | X_1 = x_1, X_2 = x_2, Y_1 = y_1, Y_2 = y_2) \\
        & \hspace{15mm} = \overline{F}^{x_1, x_2, y_1, y_2}_{\brackets{\substack{X_3\\Y_3}}}(x_2, y_2),
    \end{align*}
    due to our assumptions in (a) it follows
    \begin{align*}
        &\Prob(X_1 \leq X_2 \leq X_3, Y_1 \leq Y_2 \leq Y_3) \\
        & \hspace{15mm} = \int_{\bbr^2} \int_{[x_1, \infty[ \times [y_1, \infty[} \overline{F}^{x_1, x_2, y_1, y_2}_{\brackets{\substack{X_3\\Y_3}}}(x_2, y_2) \ d\Prob^{x_1, y_1}_{\brackets{\substack{X_2\\Y_2}}}(x_2, y_2) \ d\Prob_{\brackets{\substack{X_1\\Y_1}}}(x_1, y_1) \\
        & \hspace{15mm} \leq \int_{\bbr^2} \int_{[x_1, \infty[ \times [y_1, \infty[} \overline{F}^{x_1, x_2, y_1, y_2}_{\brackets{\substack{X^\ast_3\\Y^\ast_3}}}(x_2, y_2) \ d\Prob^{x_1, y_1}_{\brackets{\substack{X_2\\Y_2}}}(x_2, y_2) \ d\Prob_{\brackets{\substack{X_1\\Y_1}}}(x_1, y_1).
    \end{align*}
    % XXX Das obige sieht ein bisschen doof aus, finde ich...
    Furthermore,
    \begin{equation*}
        \overline{F}^{x_1, y_1}_{\brackets{\substack{\bfX\\\bfY}}^{I^c}}(x_2, x_3, y_2, y_3) \leq \overline{F}^{x_1, y_1}_{\brackets{\substack{\astbfX\\\astbfY}}^{I^c}}(x_2, x_3, y_2, y_3)
    \end{equation*}
    with $I^c = \{2, 3\}$ yields
    \begin{align*}
        \overline{F}^{x_1, y_1}_{\brackets{\substack{X_2\\Y_2}}}(x_2, y_2) 
        &= \overline{F}^{x_1, y_1}_{\brackets{\substack{\bfX\\\bfY}}^{I^c}}(x_2, -\infty, y_2, -\infty) \\
        &\leq \overline{F}^{x_1, y_1}_{\brackets{\substack{\astbfX\\\astbfY}}^{I^c}}(x_2, -\infty, y_2, -\infty) \\
        &= \overline{F}^{x_1, y_1}_{\brackets{\substack{X^\ast_2\\Y^\ast_2}}}(x_2, y_2)
    \end{align*}
    such that 
    \begin{equation*}
        \int 1_{[a, \infty[\times[b, \infty[}(x_2, y_2) \ d\Prob^{x_1, y_1}_{\brackets{\substack{X_2\\Y_2}}}(x_2, y_2) \leq \int 1_{[a, \infty[\times[b, \infty[}(x_2, y_2) \ d\Prob^{x_1, y_1}_{\brackets{\substack{X^\ast_2\\Y^\ast_2}}}(x_2, y_2)
    \end{equation*}
    for all $a, b \in \bbr$. Conditional survival functions are indeed survival functions, which can be approximated by sums of indicator functions of the form above. Hence, by bounded convergence theorem it holds
    \begin{align*}
        &\int_{\bbr^2} \int_{[x_1, \infty[ \times [y_1, \infty[} \overline{F}^{x_1, x_2, y_1, y_2}_{\brackets{\substack{X^\ast_3\\Y^\ast_3}}}(x_2, y_2) \ d\Prob^{x_1, y_1}_{\brackets{\substack{X_2\\Y_2}}}(x_2, y_2) \ d\Prob_{\brackets{\substack{X_1\\Y_1}}}(x_1, y_1) \\
        & \hspace{15mm} \leq \int_{\bbr^2} \int_{[x_1, \infty[ \times [y_1, \infty[} \overline{F}^{x_1, x_2, y_1, y_2}_{\brackets{\substack{X^\ast_3\\Y^\ast_3}}}(x_2, y_2) \ d\Prob^{x_1, y_1}_{\brackets{\substack{X^\ast_2\\Y^\ast_2}}}(x_2, y_2) \ d\Prob_{\brackets{\substack{X_1\\Y_1}}}(x_1, y_1) \\
        & \hspace{15mm} = \Prob(X_1 \leq X^\ast_2 \leq X^\ast_3, Y_1 \leq Y^\ast_2 \leq Y^\ast_3).
    \end{align*}
    By defining
    \begin{equation*}
        \widetilde{\Prob}^{x_3, y_3}(A) := \Prob(A|X^\ast_3 = x_3, Y^\ast_3 = y_3)
    \end{equation*}
    for any event $A$, in an analogous way it holds
    \begin{align*}
        &\Prob(X_1 \leq X^\ast_2 \leq X^\ast_3, Y_1 \leq Y^\ast_2 \leq Y^\ast_3) \\
        & \hspace{9mm} = \int_{\bbr^2} \widetilde{\Prob}^{x_3, y_3}(X_1 \leq X^\ast_2 \leq x_3, Y_1 \leq Y^\ast_2 \leq y_3) \ d\Prob_{\brackets{\substack{X^\ast_3\\Y^\ast_3}}}(x_3, y_3) \\
        & \hspace{9mm} = \int_{\bbr^2} \int_{]-\infty,x_3]\times]-\infty, y_3]} \widetilde{\Prob}^{x_3, y_3}(X_1 \leq x_2, Y_1 \leq \revision{y_2}|X^\ast_2 = x_2, Y^\ast_2 = y_2) \ d\widetilde{\Prob}^{x_3, y_3}_{\brackets{\substack{X^\ast_2\\Y^\ast_2}}}(x_2, y_2) \ d\Prob_{\brackets{\substack{X^\ast_3\\Y^\ast_3}}}(x_3, y_3) \\
        & \hspace{9mm} = \int_{\bbr^2} \int_{]-\infty,x_3]\times]-\infty, y_3]} \revision{F^{x_2, x_3, y_2, y_3}_{\left.\brackets{\substack{\bfX\\\bfY}}^{I^c}\right|\brackets{\substack{\astbfX\\\astbfY}}^I}(x_2, y_2)} \ d\widetilde{\Prob}^{x_3, y_3}_{\brackets{\substack{X^\ast_2\\Y^\ast_2}}}(x_2, y_2) \ d\Prob_{\brackets{\substack{X^\ast_3\\Y^\ast_3}}}(x_3, y_3) \\
        & \hspace{9mm} \leq \int_{\bbr^2} \int_{]-\infty,x_3]\times]-\infty, y_3]} 
        \revision{F^{x_2, x_3, y_2, y_3}_{\left.\brackets{\substack{\astbfX\\\astbfY}}^{I^c}\right|\brackets{\substack{\astbfX\\\astbfY}}^I}(x_2, y_2)} \ d\widetilde{\Prob}^{x_3, y_3}_{\brackets{\substack{X^\ast_2\\Y^\ast_2}}}(x_2, y_2) \ \ d\Prob_{\brackets{\substack{X^\ast_3\\Y^\ast_3}}}(x_3, y_3) \\
        & \hspace{9mm} = \Prob(X^\ast_1 \leq X^\ast_2 \leq X^\ast_3, Y^\ast_1 \leq Y^\ast_2 \leq Y^\ast_3)
    \end{align*}
    \revision{with $I = \{2, 3\}$ and $I^c = \{1\}$, accordingly.} Under the alternative assumption (b), the function
    \begin{equation*}
        \overline{H}(x_1, y_1) := \int_{[x_1, \infty[ \times [y_1, \infty[} \overline{F}^{x_1, x_2, y_1, y_2}_{\brackets{\substack{X^\ast_3\\Y^\ast_3}}}(x_2, y_2) \ d\Prob^{x_1, y_1}_{\brackets{\substack{X^\ast_2\\Y^\ast_2}}}(x_2, y_2)
    \end{equation*}
    can be, at least up to scaling, considered as a survival function, so by using an approximation by sums of indicator functions as above and inequality \eqref{eq: uncond ineq indicator fct} it follows 
    \begin{align*}
        \Prob(X_1 \leq X_2 \leq X_3, Y_1 \leq Y_2 \leq Y_3) 
        &\leq \int_{\bbr^2} \int_{[x_1, \infty[ \times [y_1, \infty[} \overline{F}^{x_1, x_2, y_1, y_2}_{\brackets{\substack{X^\ast_3\\Y^\ast_3}}}(x_2, y_2) \ d\Prob^{x_1, y_1}_{\brackets{\substack{X^\ast_2\\Y^\ast_2}}}(x_2, y_2) \ d\Prob_{\brackets{\substack{X_1\\Y_1}}}(x_1, y_1) \\
        &= \int_{\bbr^2} \overline{H}(x_1, y_1) \ d\Prob_{\brackets{\substack{X_1\\Y_1}}}(x_1, y_1) \\
        &\leq \int_{\bbr^2} \overline{H}(x_1, y_1) \ d\Prob_{\brackets{\substack{X^\ast_1\\Y^\ast_1}}}(x_1, y_1) \\
        & = \Prob(X^\ast_1 \leq X^\ast_2 \leq X^\ast_3, Y^\ast_1 \leq Y^\ast_2 \leq Y^\ast_3).
    \end{align*}
\end{proof}

Note that the assumptions on the conditional cdfs and conditional survival functions seem quite restrictive. Nevertheless, there are still natural examples of random vectors which satisfy these as we illustrate in the following. 

\begin{example}
    Let $(\bfX, \bfY)$ and $(\astbfX, \astbfY)$ be 4-dimensional random vectors defined on a common probability space with
    \begin{equation*}
        \bfX = (X_1, X_2), \quad \astbfX = (X^\ast_1, X_2), 
        \quad 
        \bfY = (Y_1, Y_2), \quad \astbfY = (Y^\ast_1, Y_2),
    \end{equation*}
    and $\bfX \overset{D}{=} \astbfX$ and $\bfY \overset{D}{=} \astbfY$, i.e., in particular it holds 
    \begin{equation*}
        X_1 \overset{D}{=} X^\ast_1 \text{ and } Y_1 \overset{D}{=} Y^\ast_1.
    \end{equation*}
    Furthermore, it is enough to suppose that $(X_1, Y_1)$ and $(X^\ast_1, Y^\ast_1)$ are both independent of $(X_2, Y_2)$, respectively. Therefore, it is sufficient to only consider the distributions of $(X_1, Y_1)$ and $(X^\ast_1, Y^\ast_1)$, respectively. \\
    Let us start with the construction of a discrete example in order to extend it to the continuous case afterwards in a natural way. \\
    Suppose the probability distributions of $(X_1, Y_1)$ and $(X^\ast_1, Y^\ast_1)$ are given by 
    \begin{align*}
        \Prob((X_1, Y_1) = (1, 3)) = \Prob((X_1, Y_1) = (2, 2)) &= \frac{1}{2} \\
        \Prob((X^\ast_1, Y^\ast_1) = (1, 2)) = \Prob((X^\ast_1, Y^\ast_1) = (2, 3)) &= \frac{1}{2}.
    \end{align*}
    In particular, we observe that $X_1 \overset{D}{=} X^\ast_1$ and $Y_1 \overset{D}{=} Y^\ast_1$. (For an overview, see Fig. \ref{fig: example restrictiveness pmf}.) 
    The cdfs as well as the survival functions are outlined in Tab. \ref{tab: example restrictiveness cmf and survival fct}. We observe
    \begin{equation*}
        F_{\brackets{\substack{X_1 \\ Y_1}}} \leq F_{\brackets{\substack{X^\ast_1 \\ Y^\ast_1}}} 
        \text{ and } 
        \overline{F}_{\brackets{\substack{X_1 \\ Y_1}}} \leq \overline{F}_{\brackets{\substack{X^\ast_1 \\ Y^\ast_1}}}
    \end{equation*}
    which implies 
    \begin{equation*}
        F_{\left.\brackets{\substack{\bfX\\\bfY}}^{I^c}\right|\brackets{\substack{\bfX\\\bfY}}^I} 
        \leq 
        F_{\left.\brackets{\substack{\astbfX\\\astbfY}}^{I^c}\right|\brackets{\substack{\bfX\\\bfY}}^I}
        \text{ and } 
        \overline{F}_{\left.\brackets{\substack{\bfX\\\bfY}}^{I^c}\right|\brackets{\substack{\bfX\\\bfY}}^I} 
        \leq 
        \overline{F}_{\left.\brackets{\substack{\astbfX\\\astbfY}}^{I^c}\right|\brackets{\substack{\bfX\\\bfY}}^I}
    \end{equation*} 
    \\
    \begin{figure}[h] %XXX das muss in jedem Fall in das Example rein und nicht vorher stehen
        \centering
        \begin{subfigure}{0.5\textwidth}
            \centering
            \begin{tikzpicture}
                \node (1st outcome) at (0, 3) {(1, 2)};
                \node (2nd outcome) at (0, 2) {(1, 3)};
                \node (3rd outcome) at (0, 1) {(2, 2)};
                \node (4th outcome) at (0, 0) {(2, 3)};

                \node (1st rv) at (-2.5, 1.5) {$\brackets{\substack{X_1 \\ Y_1}}$};
                \node (2nd rv) at (2.5, 1.5) {$\brackets{\substack{X^\ast_1 \\ Y^\ast_1}}$};

                \draw[thick, ->] (1st rv.east) -- node[above] {$\frac{1}{2}$} (2nd outcome.west);
                \draw[thick, ->] (1st rv.east) -- node[below] {$\frac{1}{2}$} (3rd outcome.west);
                \draw[thick, ->] (2nd rv.west) -- node[above] {$\frac{1}{2}$} (1st outcome.east);
                \draw[thick, ->] (2nd rv.west) -- node[below] {$\frac{1}{2}$} (4th outcome.east);
            \end{tikzpicture}
            \caption{}
        \end{subfigure}
        \begin{subfigure}{0.45\textwidth}
            \centering
            \begin{tikzpicture}
                \node (1st outcome) at (0, 3) {1};
                \node (2nd outcome) at (0, 2) {2};
                \node (3rd outcome) at (0, 1) {3};

                \node (1st rvs) at (-2, 2) {$X_1, X^\ast_1$};
                \node (2nd rvs) at (2, 2) {$Y_1, Y^\ast_1$};

                \draw[thick, ->] (1st rvs.east) -- node[above] {$\frac{1}{2}$} (1st outcome.west);
                \draw[thick, ->] (1st rvs.east) -- node[below] {$\frac{1}{2}$} (2nd outcome.west);
                \draw[thick, ->] (2nd rvs.west) -- node[above] {$\frac{1}{2}$} (2nd outcome.east);
                \draw[thick, ->] (2nd rvs.west) -- node[below] {$\frac{1}{2}$} (3rd outcome.east);
            \end{tikzpicture}
            \caption{}
        \end{subfigure}
        \caption{Illustration of the respective probability mass functions (pmfs): joint pmfs in (i) and pmfs of the respective marginals in (ii). Here, the arrows imply that the random vector of interest attains the respective outcome with the indicated probability (of $1/2$).}
        \label{fig: example restrictiveness pmf}
    \end{figure}
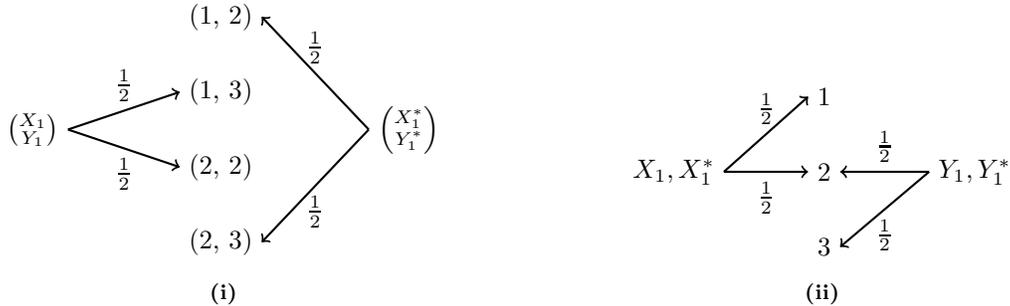
    \begin{table}[h]
        \centering
        \caption{Cdfs and survival functions of $(X_1, Y_1)$ and $(X^\ast_1, Y^\ast_1)$, respectively.}
        \vskip-.3cm
        \begin{tabular}{ccccc}
            \toprule
            $(x_1, y_1)$ &$F_{\brackets{\substack{X_1 \\ Y_1}}}$ &$F_{\brackets{\substack{X^\ast_1 \\ Y^\ast_1}}}$
            &$\overline{F}_{\brackets{\substack{X_1 \\ Y_1}}}$ &$\overline{F}_{\brackets{\substack{X^\ast_1 \\ Y^\ast_1}}}$ \\
            \midrule 
            (1, 2) &$0$ &$\frac{1}{2}$ &$1$ &$1$ \\
            (1, 3) &$\frac{1}{2}$ &$\frac{1}{2}$ &$\frac{1}{2}$ &$\frac{1}{2}$ \\
            (2, 2) &$\frac{1}{2}$ &$\frac{1}{2}$ &$\frac{1}{2}$ &$\frac{1}{2}$ \\
            (2, 3) &$1$ &$1$ &$0$ &$\frac{1}{2}$ \\
            \bottomrule
        \end{tabular}
        \label{tab: example restrictiveness cmf and survival fct}
    \end{table} 
    
    as well as
    \begin{equation*}
        F_{\left.\brackets{\substack{\bfX\\\bfY}}^{I^c}\right|\brackets{\substack{\astbfX\\\astbfY}}^I} 
        \leq 
        F_{\left.\brackets{\substack{\astbfX\\\astbfY}}^{I^c}\right|\brackets{\substack{\astbfX\\\astbfY}}^I}
        \text{ and } 
        \overline{F}_{\left.\brackets{\substack{\bfX\\\bfY}}^{I^c}\right|\brackets{\substack{\astbfX\\\astbfY}}^I} 
        \leq 
        \overline{F}_{\left.\brackets{\substack{\astbfX\\\astbfY}}^{I^c}\right|\brackets{\substack{\astbfX\\\astbfY}}^I}
    \end{equation*}
    for all $I \subsetneq \{1, 2\}$ and $I^c = \{1, 2\} \setminus I$ due to the demanded independence, therefore, both sets of assumptions in Theorem \ref{theo: OPD axiom 5 diff assump} are satisfied. \\
    All that is left is to extend the random vectors to the continuous case. For this, let $f_{\brackets{\substack{X_2\\Y_2}}}$ denote the probability density function of $(X_2, Y_2)$, and define the densities $f$ and $f^\ast$ of $(X_1, Y_1)$ and $(X^\ast_1, Y^\ast_1)$, respectively, by
    \begin{align*}
        f(x_1, y_1) 
        := f_{\brackets{\substack{X_1\\Y_1}}}(x_1, y_1)
        &= \frac{1}{2} \brackets{1_{]0,1[}(x_1) \cdot 1_{]2, 3[}(y_1) + 1_{]1, 2[}(x_1) \cdot 1_{]1, 2[}(y_1) } \\
        f^\ast(x_1, y_1) 
        := f_{\brackets{\substack{X^\ast_1\\Y^\ast_1}}}(x_1, y_1) 
        &= \frac{1}{2} \brackets{1_{]0,1[}(x_1) \cdot 1_{]1, 2[}(y_1) + 1_{]1, 2[}(x_1) \cdot 1_{]2, 3[}(y_1) }.
    \end{align*}
    Let $(X_1, Y_1)$ and $(X^\ast_1, Y^\ast_1)$ be independent of $(X_2, Y_2)$, respectively, such that the densities of $(\bfX, \bfY)$ and $(\astbfX, \astbfY)$ are given by
    \begin{align*}
        f_{\brackets{\substack{\bfX\\\bfY}}}(x_1, x_2, y_1, y_2) 
        &= f_{\brackets{\substack{X_1\\Y_1}}}(x_1, y_1) \cdot f_{\brackets{\substack{X_2\\Y_2}}}(x_2, y_2) \\
        f_{\brackets{\substack{\astbfX\\\astbfY}}}(x_1, x_2, y_1, y_2) 
        &= f_{\brackets{\substack{X^\ast_1\\Y^\ast_1}}}(x_1, y_1) \cdot f_{\brackets{\substack{X_2\\Y_2}}}(x_2, y_2).
    \end{align*}
    Obviously, it holds $X_1 \overset{D}{=} X^\ast_1$ and $Y_1 \overset{D}{=} Y^\ast_1$. Moreover, it holds $F_{\brackets{\substack{X_1\\Y_1}}} \leq F_{\brackets{\substack{X^\ast_1\\Y^\ast_1}}}$ and $\overline{F}_{\brackets{\substack{X_1\\Y_1}}} \leq \overline{F}_{\brackets{\substack{X^\ast_1\\Y^\ast_1}}}$ (see Fig. \ref{fig: example restrictiveness density}). Therefore, again both sets of assumptions in Theorem \ref{theo: OPD axiom 5 diff assump} are satisfied.%, i.e., there are random vectors which satisfy the proposed conditions, but which are not equal in distribution.
    \begin{figure}[t]
        \centering
        \begin{subfigure}{0.45\textwidth}
            \centering
            \begin{tikzpicture}[scale=0.8]
                \fill[green!5] (-1, -1) rectangle (1.5, 2.5);
                \filldraw[blue!10] (1, 1) rectangle (2, 2);
                \filldraw[blue!10] (0, 2) rectangle (1, 3);
            
                \draw[olive] (-1, 2.5) -- (1.5, 2.5) -- (1.5, -1);
                \filldraw[black] (1.5, 2.5) circle (1pt) node[anchor=west]{(1.5, 2.5)};
            
                \draw[->] (0, 0) -- (4, 0);
                \draw[->] (0, 0) -- (0, 4);
                \node at (4.2, -0.3) {$x_1$};
                \node at (-0.3, 4) {$y_1$};
                \foreach \x in {1, 2, 3}{
                    \draw (\x,0) -- (\x,-0.2);
                    \node at (\x,-0.5) {$\x$};
                    \draw (0, \x) -- (-0.2, \x);
                    \node at (-0.5, \x) {$\x$};
                };
            \end{tikzpicture}
            \caption{}
        \end{subfigure}
        \begin{subfigure}{0.45\textwidth}
            \centering
            \begin{tikzpicture}[scale=0.8]
                \fill[green!5] (-1, -1) rectangle (1.5, 2.5);
                \filldraw[blue!10] (0, 1) rectangle (1, 2);
                \filldraw[blue!10] (1, 2) rectangle (2, 3);

                \draw[olive] (-1, 2.5) -- (1.5, 2.5) -- (1.5, -1);
                \filldraw[black] (1.5, 2.5) circle (1pt) node[anchor=west]{(1.5, 2.5)};
            
                \draw[->] (0, 0) -- (4, 0);
                \draw[->] (0, 0) -- (0, 4);
                \node at (4.2, -0.3) {$x_1$};
                \node at (-0.3, 4) {$y_1$};
                \foreach \x in {1, 2, 3}{
                    \draw (\x,0) -- (\x,-0.2);
                    \node at (\x,-0.5) {$\x$};
                    \draw (0, \x) -- (-0.2, \x);
                    \node at (-0.5, \x) {$\x$};
                };
            \end{tikzpicture}
            \caption{}
        \end{subfigure}
        \caption{Illustration of the density functions $f$ and $f^\ast$: The blue squares denote the places where the respective functions hold mass. Considering the exemplary rectangle $]-\infty, 1.5] \times ]-\infty, 2.5]$ (green) we observe that $f$ has less cumulative mass than $f^\ast$. This holds for any point $(x_1, y_1) \in \bbr^2$. The same is true for rectangles of the form $[x_1, \infty[ \times [y_1, \infty[ \subset \bbr^2$.}
        \label{fig: example restrictiveness density}
    \end{figure}
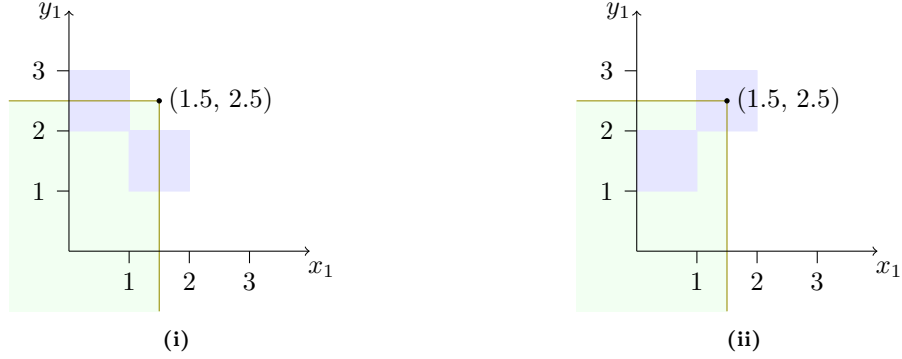
\end{example}

\pagebreak
\begin{example}
    In the previous example we assumed independence of $(X_1, Y_1)$ and $(X^\ast_1, Y^\ast_1)$ with regard to $(X_2, Y_2)$. Let us also present an example with dependent random vectors. For this, assume that $(X_2, Y_2) \in \{c_1, c_2\}$ and
    \[
    \Prob((X_2, Y_2) = c_1) = \frac{1}{2} = \Prob((X_2, Y_2) = c_2).
    \]
    We consider the dependence structure as illustrated in Fig. \ref{fig: example 2 restrictiveness pmf} (i) and (ii). The respective conditional cdfs and conditional survival functions are outlined in Tables \ref{tab: example 2 restrictiveness cmf and survival fct} (i) and (ii). We observe that the conditions on the conditional cdfs and survival functions required by the first set of assumptions are satisfied. Moreover, even the conditions on the joint distributions which are necessary for the alternative set of assumptions are fulfilled, as Fig. \ref{fig: example 2 restrictiveness pmf} (iii) and Table \ref{tab: example 2 restrictiveness cmf and survival fct} (iii) indicate. Extending this to the continuous case in an analogous manner as we have done before concludes this example.
    \begin{figure}[!ht]
        \centering
        \begin{subfigure}{1\textwidth}
            \centering
            \scalebox{0.75}{
            \begin{tikzpicture}
                \node (1st outcome) at (0, 3) {(1, 2)};
                \node (2nd outcome) at (0, 2) {(1, 3)};
                \node (3rd outcome) at (0, 1) {(2, 2)};
                \node (4th outcome) at (0, 0) {(2, 3)};

                \node (1st rv) at (-4, 1.5) {$\brackets{\substack{X_1 \\ Y_1} \Big| (X_2, Y_2) = c_1}$};
                \node (2nd rv) at (4, 1.5) {$\brackets{\substack{X^\ast_1 \\ Y^\ast_1} \Big| (X_2, Y_2) = c_1}$};

                \draw[thick, ->] (1st rv.east) -- node[above] {$\frac{1}{2}$} (2nd outcome.west);
                \draw[thick, ->] (1st rv.east) -- node[below] {$\frac{1}{2}$} (3rd outcome.west);
                \draw[thick, ->] (2nd rv.west) -- node[above] {$\frac{1}{2}$} (1st outcome.east);
                \draw[thick, ->] (2nd rv.west) -- node[below] {$\frac{1}{2}$} (4th outcome.east);
            \end{tikzpicture}}
            \caption{}
        \end{subfigure}
        \begin{subfigure}{1\textwidth}
            \centering
            \scalebox{0.75}{
            \begin{tikzpicture}
                \filldraw[white] (0, 4.7) circle (2pt);
                \node (1st outcome) at (0, 4.2) {(1, 2)};
                \node (2nd outcome) at (0, 2.8) {(1, 3)};
                \node (3rd outcome) at (0, 1.4) {(2, 2)};
                \node (4th outcome) at (0, 0) {(2, 3)};

                \node (1st rv) at (-4, 2) {$\brackets{\substack{X_1 \\ Y_1} \Big| (X_2, Y_2) = c_2}$};
                \node (2nd rv) at (4, 2) {$\brackets{\substack{X^\ast_1 \\ Y^\ast_1} \Big| (X_2, Y_2) = c_2}$};

                \draw[thick, ->] (1st rv.east) -- node[above] {$\frac{1}{4}$} (1st outcome.west);
                \draw[thick, ->] (1st rv.east) -- node[above] {$\frac{1}{4}$} (2nd outcome.west);
                \draw[thick, ->] (1st rv.east) -- node[below] {$\frac{1}{4}$} (3rd outcome.west);
                \draw[thick, ->] (1st rv.east) -- node[below] {$\frac{1}{4}$} (4th outcome.west);
                \draw[thick, ->] (2nd rv.west) -- node[above] {$\frac{1}{4}$} (1st outcome.east);
                \draw[thick, ->] (2nd rv.west) -- node[above] {$\frac{1}{4}$} (2nd outcome.east);
                \draw[thick, ->] (2nd rv.west) -- node[below] {$\frac{1}{4}$} (3rd outcome.east);
                \draw[thick, ->] (2nd rv.west) -- node[below] {$\frac{1}{4}$} (4th outcome.east);
            \end{tikzpicture}}
            \caption{}
        \end{subfigure}
        \begin{subfigure}{0.4\textwidth}
            \centering
            \scalebox{0.75}{
            \begin{tikzpicture}
                \filldraw[white] (0, 4.7) circle (2pt);
                \node (1st outcome) at (0, 4.2) {(1, 2)};
                \node (2nd outcome) at (0, 2.8) {(1, 3)};
                \node (3rd outcome) at (0, 1.4) {(2, 2)};
                \node (4th outcome) at (0, 0) {(2, 3)};

                \node (1st rv) at (-2.5, 2) {$\brackets{\substack{X_1 \\ Y_1}}$};
                \node (2nd rv) at (2.5, 2) {$\brackets{\substack{X^\ast_1 \\ Y^\ast_1}}$};

                \draw[thick, ->] (1st rv.east) -- node[above] {$\frac{1}{8}$} (1st outcome.west);
                \draw[thick, ->] (1st rv.east) -- node[above] {$\frac{3}{8}$} (2nd outcome.west);
                \draw[thick, ->] (1st rv.east) -- node[below] {$\frac{3}{8}$} (3rd outcome.west);
                \draw[thick, ->] (1st rv.east) -- node[below] {$\frac{1}{8}$} (4th outcome.west);
                \draw[thick, ->] (2nd rv.west) -- node[above] {$\frac{3}{8}$} (1st outcome.east);
                \draw[thick, ->] (2nd rv.west) -- node[above] {$\frac{1}{8}$} (2nd outcome.east);
                \draw[thick, ->] (2nd rv.west) -- node[below] {$\frac{1}{8}$} (3rd outcome.east);
                \draw[thick, ->] (2nd rv.west) -- node[below] {$\frac{3}{8}$} (4th outcome.east);
            \end{tikzpicture}}
            \caption{}
        \end{subfigure}
        \begin{subfigure}{0.4\textwidth}
            \centering
            \scalebox{0.75}{
            \begin{tikzpicture}
                \node (1st outcome) at (0, 3.1) {1};
                \node (2nd outcome) at (0, 2.1) {2};
                \node (3rd outcome) at (0, 1.1) {3};
                \node (space) at (0, 0) {};

                \node (1st rvs) at (-2, 2.1) {$X_1, X^\ast_1$};
                \node (2nd rvs) at (2, 2.1) {$Y_1, Y^\ast_1$};

                \draw[thick, ->] (1st rvs.east) -- node[above] {$\frac{1}{2}$} (1st outcome.west);
                \draw[thick, ->] (1st rvs.east) -- node[below] {$\frac{1}{2}$} (2nd outcome.west);
                \draw[thick, ->] (2nd rvs.west) -- node[above] {$\frac{1}{2}$} (2nd outcome.east);
                \draw[thick, ->] (2nd rvs.west) -- node[below] {$\frac{1}{2}$} (3rd outcome.east);
            \end{tikzpicture}}
            \caption{}
        \end{subfigure}
        \caption{Illustration of the respective probabilty mass functions (pmfs): conditional pmfs in (i) and (ii), joint pmfs in (iii) and pmfs of the respective marginals in (iv). The arrows imply that the random vector of interest attains the respective outcome with the indicated probability.}
        \label{fig: example 2 restrictiveness pmf}
    \end{figure} 
    \begin{table}[!ht]
        \centering
        \caption{(Conditional) cdfs and survival functions of $(X_1, Y_1)$ and $(X^\ast_1, Y^\ast_1)$, respectively. Here, $F^c_{\brackets{\substack{X_1 \\ Y_1}}}$ denotes the cdf of $(X_1, Y_1 | (X_2, Y_2) = c)$.}
            \vskip-.3cm
        \begin{subtable}{0.45\textwidth}
            \centering
            \begin{tabular}{ccccc}
                \toprule
                $(x_1, y_1)$ &$F^{c_1}_{\brackets{\substack{X_1 \\ Y_1}}}$ &$F^{c_1}_{\brackets{\substack{X^\ast_1 \\ Y^\ast_1}}}$
                &$\overline{F}^{c_1}_{\brackets{\substack{X_1 \\ Y_1}}}$ &$\overline{F}^{c_1}_{\brackets{\substack{X^\ast_1 \\ Y^\ast_1}}}$ \\
                \midrule 
                (1, 2) &$0$ &$\frac{1}{2}$ &$1$ &$1$ \\
                (1, 3) &$\frac{1}{2}$ &$\frac{1}{2}$ &$\frac{1}{2}$ &$\frac{1}{2}$ \\
                (2, 2) &$\frac{1}{2}$ &$\frac{1}{2}$ &$\frac{1}{2}$ &$\frac{1}{2}$ \\
                (2, 3) &$1$ &$1$ &$0$ &$\frac{1}{2}$ \\
                \bottomrule
            \end{tabular}
            \caption{}
        \end{subtable}
        \begin{subtable}{0.45\textwidth}
            \centering
            \begin{tabular}{ccccc}
                \toprule
                $(x_1, y_1)$ &$F^{c_2}_{\brackets{\substack{X_1 \\ Y_1}}}$ &$F^{c_2}_{\brackets{\substack{X^\ast_1 \\ Y^\ast_1}}}$
                &$\overline{F}^{c_2}_{\brackets{\substack{X_1 \\ Y_1}}}$ &$\overline{F}^{c_2}_{\brackets{\substack{X^\ast_1 \\ Y^\ast_1}}}$ \\
                \midrule 
                (1, 2) &$\frac{1}{4}$ &\revision{$\frac{1}{4}$} &$1$ &$1$ \\
                (1, 3) &$\frac{1}{2}$ &$\frac{1}{2}$ &$\frac{1}{2}$ &$\frac{1}{2}$ \\
                (2, 2) &$\frac{1}{2}$ &$\frac{1}{2}$ &$\frac{1}{2}$ &$\frac{1}{2}$ \\
                (2, 3) &$1$ &$1$ &$\frac{1}{4}$ &$\frac{1}{4}$ \\
                \bottomrule
            \end{tabular}
            \caption{}
        \end{subtable}
        \begin{subtable}{0.45\textwidth}
            \centering
            \begin{tabular}{ccccc}
                \toprule
                $(x_1, y_1)$ &$F_{\brackets{\substack{X_1 \\ Y_1}}}$ &$F_{\brackets{\substack{X^\ast_1 \\ Y^\ast_1}}}$
                &$\overline{F}_{\brackets{\substack{X_1 \\ Y_1}}}$ &$\overline{F}_{\brackets{\substack{X^\ast_1 \\ Y^\ast_1}}}$ \\
                \midrule 
                (1, 2) &$\frac{1}{8}$ &$\frac{3}{8}$ &$1$ &$1$ \\
                (1, 3) &$\frac{1}{2}$ &$\frac{1}{2}$ &$\frac{1}{2}$ &$\frac{1}{2}$ \\
                (2, 2) &$\frac{1}{2}$ &$\frac{1}{2}$ &$\frac{1}{2}$ &$\frac{1}{2}$ \\
                (2, 3) &$1$ &$1$ &$\frac{1}{8}$ &$\frac{3}{8}$ \\
                \bottomrule
            \end{tabular}
            \caption{}
        \end{subtable}
        \label{tab: example 2 restrictiveness cmf and survival fct}
    \end{table}
\end{example}

\newpage

%%%%%%%%%%%%%%%%%%%%%%%%%%%%%%%%%%%
 \section*{Acknowledgments} \revision{We thank the Editor-in-Chief and the Associate Editor. Furthermore, we would like to thank two anonymous referees for their valuable comments which have helped us to improve the paper.} This work was supported by the German Research Foundation (DFG), grant number SCHN 1231-3/2.

%Bibliography
\bibliographystyle{abbrv}
\nocite{*} 
\bibliography{references}

\end{document}